\documentclass[a4paper, 10pt]{amsart}
\usepackage{amsmath}
\usepackage{amsfonts}
\usepackage{amssymb}
\usepackage[utf8]{inputenc}

\usepackage{amssymb}

\usepackage{enumerate}

\usepackage{graphicx}

\newtheorem{thrm}{Theorem}[section]
\newtheorem{cor}[thrm]{Corollary}

\newtheorem{prop}[thrm]{Proposition}

\newtheorem{question}[thrm]{Question}

\theoremstyle{definition}
\newtheorem{defin}[thrm]{Definition}
\newtheorem{rem}[thrm]{Remark}

\newcommand{\CC}{\mathcal{C}}
\newcommand{\NN}{\mathbb{N}}

\newcommand{\FF}{\mathcal{F}}

\newcommand{\UU}{\mathcal{U}}
\newcommand{\VV}{\mathcal{V}}
\newcommand{\la}{\langle}
\newcommand{\ra}{\rangle}
\newcommand{\KK}{\mathcal{K}}

\begin{document}

\title[Games and hereditary Baireness]{Games and hereditary Baireness in hyperspaces and spaces of probability measures}

\author[M. Krupski]{Miko\l aj Krupski}
\address{
Institute of Mathematics\\ University of Warsaw\\ ul. Banacha 2\\
02--097 Warszawa, Poland }
\email{mkrupski@mimuw.edu.pl}

\begin{abstract}
We establish that the existence of a winning strategy in certain topological games, closely related to a strong game of Choquet,
played in a topological space $X$ and its hyperspace
$K(X)$ of all nonempty compact subsets of $X$
equipped with the Vietoris topology, is equivalent for one of the players. For a separable metrizable space $X$,
we identify a game-theoretic condition equivalent to $K(X)$ being hereditarily Baire. It implies quite easily a recent result of Gartside, Medini and Zdomskyy
that characterizes hereditary Baire property of hyperspaces $K(X)$ over separable metrizable spaces $X$ via the Menger property of the remainder of a
compactification of $X$. Subsequently, we use topological games to study hereditary Baire property in spaces of probability measures and in hyperspaces over
filters on natural numbers. To this end, we introduce a notion of strong $P$-filter $\FF$ and prove that it is equivalent to $K(\FF)$ being hereditarily Baire.
We also show that if $X$ is separable metrizable and
$K(X)$ is hereditarily Baire, then the space $P_r(X)$ of Borel probability Radon measures on $X$ is hereditarily Baire too. It follows that there exists (in ZFC)
a separable metrizable space $X$ which is not completely metrizable with $P_r(X)$ hereditarily Baire. As far as we know this is the first example of this kind. 
\end{abstract}



\maketitle

\section{Introduction}
In this paper we shall examine the interplay between certain topological games played in a topological space $X$ and its hyperspace $K(X)$ of all nonempty compact subsets
of $X$ endowed with the Vietoris topology. The
games that are considered are modifications of the celebrated Banach-Mazur game which is known to characterize the Baire property in topological spaces.
Recall that a topological space $X$ is
\textit{Baire} if the Baire Category Theorem is valid in $X$, i.e. for any countable family $\{U_n:n\in\omega\}$
consisting of nonempty open and dense subsets of $X$ the intersection $\bigcap_{n\in\omega}U_n$ is dense in $X$. A space $X$ is \textit{hereditarily Baire}
if any closed subspace of $X$ is Baire. A classical Hurewicz's theorem \cite{Hu1} asserts that, for separable metrizable spaces $X$
(in fact for first-countable regular spaces \cite{vanDouwen}, \cite{Debs}) this is equivalent
to saying that the space $\mathbb{Q}$ of rational numbers does not embed into $X$ as a closed subspace, i.e. we have:
\begin{thrm}(Hurewicz)\label{thrm_Hurewicz}
Let $X$ be a separable metrizable space. The space $\mathbb{Q}$ of rational numbers embeds into $X$ if and only if $X$ is not hereditarily Baire.  
\end{thrm}

For a separable metrizable space $X$, we identify a game-theoretic condition equivalent to $K(X)$ being hereditarily Baire,
which turns out to be useful.
For instance, using a result of Telgarsky from \cite{T},
we get the following theorem, obtained recently by Gartside, Medini and Zdomskyy \cite{GMZ}, as a corollary (the definition of a Menger space will be given later in
Section \ref{Menger}):
\begin{cor}\label{main1} (Gartside, Medini, Zdomskyy)
For a separable metrizable space $X$, the space $K(X)$ is hereditarily Baire if and only if the space $Z\setminus X$ is Menger,
where $Z$ is arbitrary compact metrizable space that contains $X$.
\end{cor}
Moreover, our approach simplifies the proof of the above statement. It follows that there are subsets of the unit interval $[0,1]$ not being $G_\delta$ with
hereditarily Baire hyperspace $K(X)$.

Interestingly, the condition '$K(X)$ is hereditarily Baire' appears quite naturally in many different contexts. In this paper we shall exhibit
two instances of this phenomenon. The first one is related to the important notion of Prohorov spaces whereas the second one deals with
filters on the set of natural numbers.

The class of Prohorov spaces
arises naturally from the following fundamental theorem in probability theory
due to Yu.V. Prohorov: In a separable completely metrizable space $X$, every compact subset $C$
of the space $P_r(X)$ of Radon probability measures on $X$ endowed with the weak topology is uniformly tight,
i.e. given $\varepsilon>0$ one can find a compact set $K\subseteq X$ such that
for every $\mu\in C$ we have $\mu(K)>1-\varepsilon$.
A topological space $X$ is called \textit{Prohorov} if every compact subset of $P_r(X)$ is uniformly tight.
So far no purely
topological description of Prohorov spaces
is known. It is also not clear if in ZFC there is a subspace of the interval $[0,1]$ which is Prohorov but not completely metrizable \cite[page 225]{Bo}.
One of the central results on (non)Prohorov spaces was proved by D. Preiss in \cite{P}. It asserts that the space $\mathbb{Q}$ of the rational numbers is not
Prohorov. Combining this result with Hurewicz's Theorem \ref{thrm_Hurewicz} we infer that
if a separable metrizable space $X$ is Prohorov, then $X$ must be hereditarily Baire. But actually more is true. It is
known that if $X$ is Prohorov then the space $P_r(X)$ is Prohorov too
(see \cite[Theorem 1]{W1}, \cite[8.10.16]{Bo}). Hence, if a separable metrizable space $X$ is Prohorov, then $P_r(X)$ must be hereditarily Baire. This calls for research on
(hereditary) Baireness of the space of measures $P_r(X)$. This topic was studied in the 1970s and 1980s (see e.g. \cite{B}, \cite{BC}).
But since then no substantial
progress in this area was made and we still don't have the full picture of completeness type properties in spaces of measures.
We will show that if the hyperspace $K(X)$ is hereditarily Baire then $P_r(X)$ is also hereditarily Baire.
We conclude that there is (in ZFC) a subset of $[0,1]$ which is not $G_\delta$ and whose space of Borel probability Radon measures is hereditarily Baire.

Our second application of game-theoretic considerations on hereditarily Baire hyperspaces, is concerned with filters on the set $\mathbb{N}$ of natural numbers.
We shall study hereditary Baireness of $K(\mathcal{F})$, where $\mathcal{F}$ is a filter on $\mathbb{N}$
considered as a subspace of the Cantor set $\{0,1\}^\mathbb{N}$. This leads naturally to the notion of a strong $P$-filter that was defined by
Laflamme \cite{La} for ultrafilters but never considered in the general setting of filters.\footnote{In \cite{BHV} and \cite{GHM} the authors consider a similar notion of strong
$P^+$-filter; both notions agree for ultrafilters but in general they are different.} We find a natural context for this notion.
It turns out that a filter on $\mathbb{N}$ is a strong $P$-filter if and only if the hyperspace $K(\FF)$ is hereditarily Baire
if and only if $\{0,1\}^\mathbb{N}\setminus \FF$ is Menger (see Theorem \ref{filters}).
Combining this with a recent result of Bella and Hern\'{a}ndez-Guti\'{e}rrez
\cite{BH-G}
we obtain a nice counterpart of a theorem of Marciszewski from \cite{Marciszewski}, cf. Theorem \ref{filters1}.

It is worth mentioning that our results
give insight to the line of research in topology initiated by Henriksen and Isbell \cite{HI} where one investigates
properties of $X$ by looking at the remainder $bX\setminus X$ of a compactification $bX$ of $X$ (cf. Corollary \ref{main1}).
Let $\mathcal{P}$ be a topological property invariant under images and preimages by perfect maps (e.g. the Menger property). We say that a topological space
$X$ is \textit{$\mathcal{P}$ at infinity} if for some, equivalently for every, compactification $bX$ of $X$, the remainder
$bX\setminus X$
has the property $\mathcal{P}$. For instance, it is well known that $\sigma$-compactness
at infinity is precisely completeness in the sense of \v{C}ech. Heriksen and Isbell \cite{HI} characterized spaces that are Lindel\"of at infinity as those
of countable type.
Mengerness at infinity (which lies between $\sigma$-compactness at infinity and Lindel\"ofness at infinity) was studied by  
several authors, e.g. \cite{AB}, \cite{BZT}, \cite{BH-G}.
We will show
that every space (not necessarily metrizable) that is Menger at infinity must be Baire, which settles a question of M.\ Sakai
\cite[Question 2.8]{AB}.

\section{Notation}

\subsection{Games}
As we have already mentioned, all games that are considered in this paper are modifications of the
Banach-Mazur game.
Recall that, for a topological space $X$, the
\textit{Banach-Mazur game} $BM(X)$ on $X$ is a game with $\omega$-many innings, played by two players: player I and player
II, who alternately choose sets $U_0,U_1,U_2,\ldots$ that are open in $X$ and satisfy
$U_0\supseteq U_1\supseteq U_2\supseteq\ldots$. Player I wins the run of the game if $\bigcap_{n\in\omega}U_n=\emptyset$.
It is well known (cf. e.g., \cite[Theorem 8.11]{Kechris}) that the Banach-Mazur game characterizes the property of being Baire
in the following way:
\begin{thrm}\label{Oxtoby}
(Oxtoby) A nonempty topological space $X$ is Baire if and only if player  I has no winning strategy in the Banach-Mazur
game $BM(X)$.
\end{thrm}

Let $Z$ be a space and let $X\subseteq Z$ be a subspace of $Z$.

The \textit{strong Choquet game} on $Z$ with values in $X$ is a game with $\omega$-many innings, played alternately
by two players: I and II.
Player I begins the game and makes the first move by choosing
a pair $(x_0,U_0)$, where $x_0\in X$ and $U_0$ is an open neighborhood of $x_0$ in $Z$. Player II responds by choosing an open
(in $Z$) set $V_0$ such that
$x_0\in V_0\subseteq U_0$. In the second round of the game, player I picks a pair $(x_1,U_1)$, where
$x_1\in V_0$ and $U_1$ is an open subset of $Z$ with
$x_1\in U_1\subseteq V_0$. Player II responds by picking an open (in $Z$) set $V_1$ such that $x_1\in V_1\subseteq U_1$. The game continues in this way and stops
after $\omega$ many rounds (cf. Figure \ref{fig:fig1}). Player II wins the game if
$\left(\bigcap_{n\in\omega} U_n\right)\cap X\neq\emptyset$. Otherwise, I wins.

The game described above is denoted by $Ch(Z,X)$ and by $Ch(X)$ we will denote the game $Ch(X,X)$.

\medskip

\begin{figure*}[h]
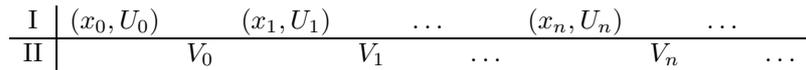

	\centering
\begin{tabular}{c|c c c c c c c c c c}
I  & $(x_0,U_0)$&  & $(x_1,U_1)$& & $\ldots$  & & $(x_n,U_n)$ & & $\ldots$  \\
\hline
II &       & $V_0$ & &$V_1$ & & $\ldots$ & & $V_n$ & & $\ldots$ \\
\end{tabular}
\caption{The strong Choquet game}
\label{fig:fig1}
\end{figure*}
\medskip

It is known (see \cite{vanDouwen}, \cite{Telgarsky}, \cite{Debs}) that, for first-countable regular spaces (in particular for all metrizable spaces),
the strong Choquet game characterizes the property of being hereditarily Baire in the following way:

\begin{thrm}\label{thrm:characterization}
Let $X$ be a first-countable regular space. Then $X$ is hereditary Baire if and only if player I has no winning strategy in the
strong Choquet game $Ch(X)$.
\end{thrm}	

Let us describe now some known modifications of the game $Ch(Z,X)$ that are discussed in the present paper. Our terminology
follows \cite{T}.

As above, $Z$ is a space and $X$ is a subspace of $Z$. The \textit{Porada game} on $Z$ with values in $X$ is
denoted by $P(Z,X)$ and is played as the strong Choquet game
$Ch(Z,X)$ with the only difference that player II wins the game $P(Z,X)$
if $\emptyset\neq \bigcap_{n\in\omega} U_n \subseteq X$ and otherwise player I wins. The game $P(Z,X)$ was introduced in \cite{Po}.
The \textit{$k$-Porada game} $kP(Z,X)$ on $Z$ with values in $X$ is a modification of the Porada game
$P(Z,X)$ in which player I instead of points $x_n\in X$ and their neighborhoods $U_n$, picks compact sets $K_n\subseteq X$
and open sets $U_n\subseteq Z$ with $K_n\subseteq X\cap U_n$. Player II responds by choosing sets $V_n$ open in $Z$ satisfying
$K_n\subseteq V_n \subseteq U_n$ (cf. Figure \ref{fig:fig2}). The winning condition is the same
as in the Porada game, i.e. player II wins the game $kP(Z,X)$ if $\emptyset\neq \bigcap_{n\in\omega} U_n \subseteq X$
and otherwise player I wins.

\begin{figure*}[h]
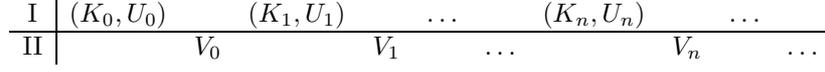

	\centering
\begin{tabular}{c|c c c c c c c c c c c }
I  & $(K_0,U_0)$&  & $(K_1,U_1)$& & $\ldots$  & & $(K_n,U_n)$ & & $\ldots$ \\
\hline
II &       & $V_0$ & &$V_1$ & & $\ldots$ & & $V_n$ & & $\ldots$ \\
\end{tabular}
\caption{The $k$-Porada game}
\label{fig:fig2}
\end{figure*}

\subsection{Strategies} For topological spaces $X\subseteq Z$, let $G(Z,X)$ be one of the games $Ch(Z,X)$ or $P(Z,X)$. Denote by
$\tau_Y$ the collection of all nonempty open subsets of the space $Y$.
A \textit{strategy} of player I in the game $G(Z,X)$ is a map $s$ defined inductively as follows:
$s(\emptyset)\in X\times \tau_Z$. If the strategy $s$ is defined for the first $n$ moves then an $n$-tuple $(V_0,V_1,\ldots,V_{n-1})\in \tau_Z^n$ is called
\textit{admissible} if $x_0\in V_0\subseteq U_0$ and $x_i\in V_i\subseteq U_i$, where $i\in\{1,\ldots , n-1\}$ and $(x_i,U_i)=s(V_0,\ldots V_{i-1})$. For any
admissible $n$-tuple $(V_0,\ldots, V_{n-1})$ we choose a pair $(x_n,U_n)\in V_{n-1}\times \tau_{V_{n-1}}$.
A strategy $s$ of player I in the game $G(Z,X)$ is called \textit{winning} if player I wins every run of the game $G(Z,X)$ in which she plays according to the
strategy $s$.

In a similar manner we define strategies, and winning strategies in the games $kP(Z,X)$ and $BM(X)$.

\begin{rem}\label{uwaga}
Let $X\subseteq Z$ be topological spaces. Note that if we consider the following conditions:
\begin{enumerate}[(i)]
 \item Player I has a winning strategy in the Banach-Mazur game $BM(X)$;
 \item Player I has a winning strategy in the strong Choquet game $Ch(Z,X)$;
 \item Player I has a winning strategy in the Porada game $P(Z,X)$;
 \item Player I has a winning strategy in the k-Porada game $kP(Z,X)$;
\end{enumerate}
then $(i)\Rightarrow (ii)\Rightarrow (iii)\Rightarrow (iv)$.
\end{rem}

\subsection{Menger Spaces}\label{Menger} Recall that a topological space $X$ is called \textit{Menger}
if for every sequence $(\UU_n)_{n\in \omega}$ of open covers of $X$, there is a sequence $(\VV_n)_{n\in\omega}$
such that for every $n$, $\VV_n$ is a finite subfamily of $\UU_n$ and $\bigcup_{n\in\omega} \VV_n$ covers $X$.
It is well known that the Menger property may be characterized in terms of games in the following way
(see e.g. \cite[Theorem 13]{Scheepers} or \cite[Theorem 2.32]{AD}, cf. \cite[Satz X]{Hu2}):
For a topological space $X$, let $M(X)$ be a two player game in which player I begins and in the $n$-th move chooses
an open cover $\mathcal{U}_n$ of $X$. Player II responds by picking in his $n$-th move a finite subcollection
$\mathcal{V}_n\subseteq\mathcal{U}_n$. Player II wins the run of the game if $\bigcup_n\mathcal{V}_n$ is a cover of $X$.
Otherwise, player I wins.
The game $M(X)$ is called the \textit{Menger game}.

\begin{thrm}\label{Menger-charakteryzajca}(Hurewicz \footnote{In \cite{AD} the authors attribute this result to Witold Hurewicz \cite{Hu2}. Although Hurewicz did not consider games in \cite{Hu2},
his proof of \cite[Satz X]{Hu2} can easily be expressed in the language of games, cf. \cite[Theorem 13]{Scheepers}.})
A topological space $X$ is Menger if and only if player I does not have a winning strategy in the Menger game $M(X)$.
\end{thrm}

Clearly, $$\sigma\text{-compact}\Rightarrow \text{Menger}\Rightarrow \text{Lindel\"of}.$$
The space $\omega^\omega$ is Lindel\"of but not Menger. There are (in ZFC) separable metrizable Menger spaces which are not
$\sigma$-compact, cf. \cite{FM}, \cite{BT}

\subsection{Vietoris topology}
If $X$ is a topological space, then by $K(X)$ we denote the space of all nonempty compact subsets of $X$ equipped with the Vietoris topology, i.e. basic open sets in $K(X)$ are of the form
$$\langle U_0,\ldots , U_n  \rangle=\{K\in K(X) : K\subseteq \bigcup_{i=0}^n U_i\quad \text{and}\quad K\cap U_i\neq \emptyset,\text{ for every } i\leq n\},$$
where each $U_i$ is an open subset of $X$.

The proposition below is a part of folklore. 

\begin{prop}\label{folklor}
Let $\mathcal{C}$ be a compact subset of $K(X)$, i.e., $\mathcal{C}$ is a family of compact subsets of $X$ which itself is compact in the Vietoris topology on $K(X)$.
Then $\bigcup \mathcal{C}$ is a compact subset of $X$.
\end{prop}
\begin{proof}
Let $\mathcal{U}=\{U_t:t\in T\}$ be an open cover of $\bigcup\mathcal{C}$. Consider the following collection $\mathcal{W}$ of open subsets of $K(X)$:
$$\mathcal{W}=\{\la U_t:t\in F\ra:F\subseteq T\text{ is finite}\}.$$
We claim that $\mathcal{W}$ covers $\mathcal{C}$. Indeed, take $A\in \mathcal{C}\subseteq K(X)$. We have $A\subseteq \bigcup\mathcal{C}$ and thus $\UU$ is an open
cover of $A$. By compactness of $A$, there is a finite $F\subseteq T$ such that $A\subseteq \bigcup\{U_t:t\in F\}$ and $A$ meets each $U_t$, for $t\in F$.
It follows that $A\in \la U_t:t\in F\ra$, so $\mathcal{W}$ covers $\mathcal{C}$.

Now, since $\mathcal{C}$ is compact, there is a finite collection $\{F_1,\ldots , F_n\}$ of finite subsets of $T$ such that the family
$$\left\{\la U_t:t\in F_i\ra:i=1,\ldots ,n\right\}$$
covers $\mathcal{C}$. It follows that $\{U_t:t\in F_1\cup\ldots\cup F_n\}\subseteq \UU$ is a finite subcover of $\bigcup \mathcal{C}$.
\end{proof}

\section{Equivalence of certain games}

In this section we show that if $X$ is a subspace of a compact Hausdorff space $Z$ then, from player's one viewpoint, the Porada game
$kP(Z,X)$ is equivalent to the
Porada game $P(K(Z), K(X))$ and equivalent to the $k$-Porada game $kP(K(Z),K(X))$. Note that the latter two games are played
in the hyperspace $K(Z)$ and have values in the hyperspace $K(X)$.

Now, suppose that $Z$ is compact metric space containing $X$.
It is known and easy to verify that in this case the games
$Ch(X)$, $Ch(Z,X)$ and
$P(Z,X)$ are equivalent in the sense that player $Y$ has a winning strategy in one of the games if and only if player $Y$ has a winning strategy in all of the games.
In particular, $X$ is hereditarily Baire if and only if player I has no winning strategy in the Porada game $P(Z,X)$. Of course, if $Z$ is compact metric, then
$K(Z)$ is compact metric too. Thereby, we establish that
the game-theoretic condition 'player I has no winning strategy in the $k$-Porada game $kP(Z,X)$' characterizes hereditary Baireness of
the hyperspace $K(X)$.

Corollary \ref{main1} will follow immediately from this assertion and the following 1984 result of Telgarsky:
\begin{thrm}\cite[Theorem 2]{T}\label{thrm_Telgarsky}
if $X\subseteq Z$, where $Z$ is compact (not necessarily metrizable),
then the $k$-Porada game $kP(Z,X)$ is equivalent
(for both players) to the Menger game $M(Z\setminus X)$. 
\end{thrm}

Before proceeding to the proof of the main theorem of this section let us make two comments.
\begin{rem}\label{remark1}
When constructing winning strategies for player I in a Porada game
we can always assume that player II picks only basic open sets. Indeed, if $V_i$ is the $i$-th move of player II and $V_i$ is not basic, then instead of
applying her strategy directly to $V_i$, player I can first pick a basic open set $V'_i$ such that the sequence $(V_0,\ldots, V_{i-1}, V'_i)$ is admissible
(cf. section 1.3) and then apply her strategy to this sequence.
\end{rem}
\begin{rem}\label{remark2}
Suppose that $s$ is a winning strategy of player I in either of the games: Porada, $k$-Porada, strong Choquet. There is another winning strategy $s'$ of player I
such that the open set $U_n$ in the pair $s'(V_0,\ldots ,V_{n-1})$ satisfies
$\overline{U}_n\subseteq V_{n-1}$ (i.e. the closure of $U_n$ is contained in $V_{n-1}$).
\end{rem}

\begin{thrm}\label{main}
Let $X$ be a subspace of a compact Hausdorff space $Z$. The following conditions are equivalent:
\begin{enumerate}[(a)]
	\item Player I has a winning strategy in the game $kP(Z,X)$.
	\item Player I has a winning strategy in the game $P(K(Z),K(X))$.
	\item Player I has a winning strategy in the game $kP(K(Z),K(X))$.
	\item $Z\setminus X$ is not Menger.
	\item $K(Z)\setminus K(X)$ is not Menger.
\end{enumerate}
\end{thrm}

\begin{proof}
Suppose that $s$ is a winning strategy for player I in $kP(Z,X)$. We shall construct a winning strategy $t$ for
player I in $P(K(Z),K(X))$. Let $t(\emptyset)=(K_0,\la U_0\ra)$ where $(K_0,U_0)=s(\emptyset)$.
Let $V^0$ be the first move of II in $P(K(Z),K(X))$, i.e. the set $V^0$ is open in $K(Z)$ and $K_0\in V^0\subseteq \la U_0\ra$.
Without loss of generality we may assume that $V^0$ is basic, i.e. $V^0=\la V^0_1,\ldots , V^0_{n(0)} \ra$, for some natural number
$n(0)$ and open sets $V^0_i$, $i=1,\ldots, n(0)$, cf. Remark \ref{remark1}.
Let $V_0= V^0_1\cup\ldots \cup V^0_{n(0)}$. Since $K_0\in V^0\subseteq \la U_0\ra$, we have $K_0\subseteq V_0\subseteq U_0$ and hence
$V_0$ is a legal move of player II in the game $kP(Z,X)$.

We have $s(V_0)=(K_1,U_1)$ for some $K_1\in K(X)$ and $U_1$ open in $Z$ with $K_1\subseteq U_1\subseteq\overline{U_1}\subseteq V_0\subseteq U_0$,
cf. Remark \ref{remark2}.
Let
$$t(V^0)=(K_0\cup K_1,\la V^0_1,\ldots , V^0_{n(0)}, U_1 \ra).$$
Now, player II responds by picking a basic open $V^1=\la V^1_1,\ldots , V^1_{n(1)}\ra$ with
$K_0\cup K_1\in V^1\subseteq t(V^0)$.
Let $$V_1= \left(V^1_1\cup\ldots \cup V^1_{n(1)}\right)\cap U_1.$$
Obviously, $K_1\subseteq V_1\subseteq U_1$ which means that $V_1$ is a legal move of player II in $kP(Z,X)$.

We have $s(V_0,V_1)=(K_2,U_2)$ for some $K_2\in K(X)$ and $U_2$ open in $Z$ with $K_2\subseteq U_2\subseteq \overline{U_2}\subseteq V_1$.
Let
$$t(V^0,V^1)=(K_0\cup K_1\cup K_2,\la V^1_1,\ldots , V^1_{n(1)}, U_2 \ra)$$ and so on.

We claim that $t$ constructed in the way described above is a winning strategy for player I in $P(K(Z),K(X))$. To this end, it suffices to show that
\begin{equation}\label{warunek}
\bigcap_k \la V^k_1,\ldots , V^k_{n(k)}\ra\nsubseteq K(X).
\end{equation}
Since $s$ is a winning strategy for I in $kP(Z,X)$, it follows that
$K=\bigcap_n U_n =\bigcap_n \overline{U_n}\in K(Z)\setminus K(X)$.
For every $n$, we have $K_n\subseteq U_n$, so
the set $\bigcup_{n=0}^\infty K_n\cup K$ is compact and does not belong to $K(X)$ (because $K\notin K(X)$).
Moreover,
$$\bigcup_{n=0}^\infty K_n\cup K\in \bigcap_k \la V^k_1,\ldots , V^k_{n(k)}\ra,$$
which proves \eqref{warunek}. This finishes the proof of the implication $(a)\Rightarrow (b)$.

\medskip

The implication $(b)\Rightarrow (c)$ is trivial.

\medskip

Let us show that $(c)\Rightarrow (a)$.
We need to fix some notation first. In each round of the game $kP(K(Z),K(X))$, player I picks a compact subset of the hyperspace $K(X)$ together with an open
neighborhood $U\subseteq K(Z)$ of it. We will denote the $i$-th move of I
by $(\KK_i,U^i)$, i.e. the script letter  $\KK_i$ will stand for a compact subset of the hyperspace $K(X)$. Note that by Proposition \ref{folklor}, $\bigcup \KK_i$
is a compact subset of $X$.
It is not difficult to see that if player I has a winning
strategy in $kP(K(Z), K(X))$ she also has a winning strategy in this game where open sets $U^i$ are finite unions of basic open sets.
Indeed it suffices to replace $U^i$ by an arbitrary
union of basic sets containing $\KK_i$ and contained in $U^i$.
Therefore, without loss of generality we may assume that always
$$U^i=\bigcup_{j=1}^{m(i)}\la U^i_{j1},\ldots , U^i_{js(j)}\ra,$$
for some natural numbers $m(i),s(1), s(2) ,\ldots, s(m(i))$.
By $U_i$ we will denote the union
\begin{equation}\label{suma}
U_i=\bigcup_{j=1}^{m(i)}\bigcup_{k=1}^{s(j)}U^i_{jk}.
\end{equation}

In his $i$-th move, player II picks an open neighborhood $V^i$ of $\KK_i$ in $K(Z)$ such that $V^i\subseteq U^i$. It is easy to observe that without loss of
generality we may assume that
\begin{equation}\label{domkniecie}
\overline{V^i}\subseteq U^i
\end{equation}
too (where the closure is taken in $Z$).

\medskip

Now, suppose that
$s$ is a winning strategy of player I in the game $kP(K(Z),K(X))$. We shall construct a winning strategy $t$ for player I in the game $kP(Z,X)$.

First, we define $t(\emptyset)=\left(\bigcup \KK_0,U_0 \right)$, where $(\KK_0,U^0)=s(\emptyset)$. It follows from \eqref{suma} and Proposition \ref{folklor}
that $t(\emptyset)$ defined in this way is a legal move of player I in the game $kP(Z,X)$.
Let $V_0$ be the first move of II in $kP(Z,X)$, i.e.
$V_0$ is an open set in $Z$ containing $\bigcup \KK_0$ and contained in $U_0$.
Let $V^0=\la V_0\ra \cap U^0$. Clearly, $\KK_0\subseteq V^0\subseteq U^0$ and $V^0$ is an open set in $K(Z)$.
We let
$$t(V_0)=\left(\bigcup \KK_1,U_1\right)\text{, where} \quad (\KK_1,U^1)=s(V^0) \quad\text{ (cf. \eqref{suma}).}$$
Player II responds by picking and open set $V_1\subseteq Z$ such that $\bigcup \KK_1\subseteq V_1\subseteq U_1$.
We let $V^1=\la V_1 \ra\cap U^1$ and define
$$t(V_0, V_1)=\left(\bigcup \KK_2,U_2\right)\text{, where} \quad (\KK_2,U^2)=sV_0, V^1).$$
The game continues following this pattern.

The space $K(Z)$ is compact so by \eqref{domkniecie} we have $\bigcap_{n=0}^{\infty} U^n=\bigcap_{n=0}^{\infty} \overline{V^n}\neq \emptyset$.
This, together with the fact that
$s$ is a winning strategy of player I in $kP(K(Z),K(X))$, gives
$$\bigcap_{n=0}^{\infty} U^n\nsubseteq K(X).$$
Hence, there is a compact set $C\in K(Z)\setminus K(X)$ such that
$C\subseteq \bigcap_{n=0}^{\infty} U^n$. In particular, $C\in \bigcap_{n=0}^\infty U_n$ and thus
$\bigcap_{n=0}^\infty U_n\nsubseteq X$. So player I wins the game $kP(Z,X)$ if playing according to $t$.

\medskip

We have proved $(a)\Leftrightarrow (b)\Leftrightarrow (c)$.
The equivalences $(a)\Leftrightarrow (d)$ and $(c)\Leftrightarrow (e)$ follow directly from Theorems \ref{thrm_Telgarsky} and \ref{Menger-charakteryzajca}.
\end{proof}

Now we can easily conclude the result established by Gartside Medini and Zdomskyy in \cite{GMZ}, cf. Corollary \ref{main1}.

\begin{cor}\label{maincor}
Consider the following properties of a topological space $X$:
\begin{enumerate}[(A)]
 
 \item $X$ is Menger at infinity.
 \item $K(X)$ is hereditarily Baire.
\end{enumerate}
Then if $K(X)$ is first-countable and regular, then  $(A)\Rightarrow (B)$. If $X$ is separable metrizable then
$(A)\Leftrightarrow (B)$
\end{cor}
\begin{proof}
Suppose that $K(X)$ is first-countable regular and $bX$ is a compactification of $X$ with $bX\setminus X$ Menger. By Theorem \ref{main}, player I has no winning
strategy in the Porada game $P(K(bX),K(X))$. It follows from Remark \ref{uwaga} and Theorem \ref{thrm:characterization} that $K(X)$ is hereditarily Baire.
To prove the second part, let us suppose that $X$ is separable metrizable and $K(X)$ is hereditarily Baire. Let $Z$ be a metric compactification of $X$.
It is known and easy to verify that the games
$Ch(X)$, $Ch(Z,X)$ and
$P(Z,X)$ are equivalent in the sense that player $Y$ has a winning strategy in one of the games if and only if player $Y$ has a winning strategy in all of the games.
Thus, combining Theorem \ref{thrm:characterization} and Theorem \ref{main} we conclude that $Z\setminus X$ is Menger.
\end{proof}

\begin{cor}\label{maincor2}
There exists (in ZFC) a non-Polish space $X\subseteq [0,1]$ with $K(X)$ hereditarily Baire.
\end{cor}
\begin{proof}
Let $X=[0,1]\setminus M$, where $M$ is a Menger space which is not $\sigma$-compact \cite{FM}, \cite{BT}.
Then $X$ is not Polish and $K(X)$ is hereditarily Baire by
Theorem \ref{main}
\end{proof}

M.~Sakai noticed (see \cite[Theorem 2.9]{AB}) that if a first-countable Tychonoff space is Menger at infinity, then it must be hereditarily Baire. It was not clear
however whether any (not necessarily first-countable) Menger at infinity space must be Baire \cite[Question 2.8]{AB}. The affirmative answer to this question
follows immediately from Theorem \ref{main}. We have:

\begin{cor}
For any Tychonoff space $X$ if $X$ is Menger at infinity, then $X$ is Baire. 
\end{cor}
\begin{proof}
Let $bX$ be a compactification of $X$ with Menger remainder $bX\setminus X$. From Theorem \ref{main} we infer that player I has no winning strategy in the $k$-Porada
game $kP(bX,X)$. Hence, by Remark \ref{uwaga} player I has no winning strategy in the Banach-Mazur game $BM(X)$. It follows from Theorem \ref{Oxtoby} that $X$ is
Baire.
\end{proof}

\section{Filters}

An interesting class of metric spaces is formed by filters on the naturals. Recall that a family $\FF$ of subsets of $\mathbb{N}$ is called \textit{filter}
if it satisfies the following conditions:
\begin{enumerate}
 
 \item $\emptyset\notin \FF$ and $\NN\in \FF$
 \item if $A\in \FF$ and $A\subseteq B$ then $B\in \FF$
 \item if $A,B\in \FF$ then $A\cap B\in \FF$
 
\end{enumerate}
We will consider only free filters, i.e. filters containing all cofinite subsets of $\mathbb{N}$. If $\FF\subseteq P(\NN)$ is a filter, the collection
$\FF^\ast=\{A\subseteq \NN:\NN\setminus A\in \FF\}$ is its dual ideal and $\FF^+=P(\NN)\setminus \FF^\ast$ is the family of $\FF$-positive sets.

Identifying a set with its characteristic function, we treat a filter $\FF$ as a subspace of the Cantor set $2^\NN$. Filters correspond to countable spaces with
precisely one non-isolated point: Given a filter $\FF$ on $\mathbb{N}$, by $N_\FF$ we denote the space $\NN\cup\{\infty\}$ where
points of $\NN$ are isolated and the family
$\{A\cup\{\infty\}:A\in\FF\}$ is a
neighborhood base at $\infty$.

Recall that a filter $\FF$ on $\NN$ is a \textit{$P$-filter} if for every decreasing sequence $(A_n)_{n\in\omega}$ of sets from $\FF$, there exists $A\in\FF$
such that $A$ is almost contained in each $A_n$, i.e. for every $n\in\omega$ the set $A\setminus A_n$ is finite. $P$-ultrafilters are called
\textit{$P$-points}. A filter is \textit{(non)meager} is it is
(non)meager as a subset of the Cantor set $2^\mathbb{N}$.

Since $\FF$ is a metric space contained naturally in the compact metric space $2^\NN$, one may ask whether it is possible to characterize
filters $\FF$ for which $K(\FF)$ is hereditarily Baire. This question is motivated by the following result of Marciszewski \cite[Theorem 1.2]{Marciszewski},
that gives a characterization of hereditarily Baire filters.

\begin{thrm}(Marciszewski)\label{Marciszewski}
Let $\FF$ be a filter on $\NN$. The following conditions are equivalent.
\begin{enumerate}
\item $\FF$ is a nonmeager $P$-filter
\item $\FF$ is hereditarily Baire
\item $C_p(N_\FF)$ is hereditarily Baire
\end{enumerate}
\end{thrm}

The theorem above not only gives a description of hereditarily Baire filters in terms of a property of filters, but also sheds some light on hereditary Baireness
of function spaces. Recall that for a Tychonoff space $X$, by $C_p(X)$ we denote the space of all continues real valued functions on $X$ endowed with the
topology of pointwise convergence.
It is known that the Baire property of $C_p(X)$ can be characterized in terms of the topology of $X$. However, no such characterization is known
for hereditarily Baire
$C_p(X)$. Theorem \ref{Marciszewski} gives a complete characterization of hereditarily Baire function spaces over countable spaces with one non-isolated point.

Let us remark that it is still not clear
if there exists in ZFC a non-discrete countable space $X$ for which $C_p(X)$ is hereditarily Baire. There are however (in ZFC) uncountable spaces of that sort.
It follows from Proposition in Section 3.1 of \cite{CP} that if $\Gamma$ is an uncountable set then the space $C_p(\Gamma_\FF)$ is hereditarily Baire, where $\FF$
is the filter of co-countable subsets of $\Gamma$.

The following notion was introduced by Laflamme in \cite{La} for ultrafilters.
\begin{defin}
A filter $\FF\subseteq P(\NN)$ is called a \textit{strong $P$-filter} if for any sequence $(\CC_{n})_{n\in\NN}$ of compact subsets of $\FF$, there is a sequence
$0=k_0<k_1<k_2<\ldots$ of naturals, such that
$$\bigcup_{n\in\NN}\left(X_n\cap [k_n,k_{n+1})\right)\in\FF,$$
whenever $X_n\in \CC_n$, for each $n\in\NN$
\end{defin}
A strong $P$-ultrafilter is called a \textit{strong $P$-point}.
It seems that strong $P$-filters were never studied in the literature; Guzm\'{a}n et al. \cite{GHM}, Blass et al. \cite{BHV} considered a
similar notion of strong $P^+$-filters.
It is not clear whether strong $P$-filters exist in ZFC. Observe that a strong $P$-filter $\FF$ is a $P$-filter. To see this, take a decreasing sequence $(A_n)$
of sets from $\FF$. Applying the definition of strong $P$-filter to the sequence of singletons $\{A_n\}$ one gets a sequence
$0=k_0<k_1<k_2<\ldots$ of naturals such that the set
$$A=\bigcup_{n\in\NN}\left(A_n\cap [k_n,k_{n+1})\right)$$
belongs to $\FF$. Now note that $A\setminus A_n \subseteq [0,k_n)$ because the sequence $(A_n)_{n\in \omega}$ is decreasing, and hence $A$
is almost contained in each $A_n$.
Combining the next theorem with Theorem \ref{Marciszewski} we infer that strong $P$-filters are also nonmeager. It is known (see \cite{BHV}) that consistently,
there are $P$-points (i.e. $P$-ultrafilters) which are not strong $P$-points.

We will prove the following counterpart of Theorem \ref{Marciszewski}:

\begin{thrm}\label{filters1}
Let $\FF$ be a filter on $\NN$. The following conditions are equivalent.
\begin{enumerate}
\item $\FF$ is a strong-P-filter
\item $K(\FF)$ is hereditarily Baire
\item $K(C_p(N_\FF))$ is hereditarily Baire
\end{enumerate}
\end{thrm}

The equivalence $(2)\Leftrightarrow (3)$ follows from Corollary \ref{maincor} and a recent result of Bella and Hern\'{a}ndez-Guti\'{e}rrez who proved in \cite{BH-G}
that a filter $\FF$ is Menger at infinity if and only if the space $C_p(N_\FF)$ is Menger at infinity. Thus to prove Theorem \ref{filters1} it remains to show:
\begin{thrm}\label{filters} Let $\FF$ be a filter on $\mathbb{N}$.
The following conditions are equivalent.
\begin{enumerate}
    \item $\FF$ is a strong $P$-filter
    \item $\FF$ is Menger at infinity
    \item $K(\FF)$ is hereditarily Baire
    \item $\FF^+$ is Menger
\end{enumerate}
\end{thrm}
\begin{proof}
It follows from Corollary \ref{maincor} that $(2)\Leftrightarrow (3)$. Since $\FF^+$ is homeomorphic to $P(\NN)\setminus \FF$ we get $(2)\Leftrightarrow (4)$.
Thus it it enough to show
$(1)\Rightarrow (2)$ and $(3)\Rightarrow (1)$.
To prove that $(1)\Rightarrow (2)$, let us suppose that a filter $\FF$ is strong $P$.
Let $(\mathcal{U}_n)_{n\in \NN}$ be a sequence of open covers of $P(\NN)\setminus \FF$. For each $n\in \NN$ let 
$$\CC_n=P(\NN)\setminus \bigcup\UU_n.$$
Note that $\CC_n$ is compact and $\CC_n\subseteq \FF$ because the family $\UU_n$ covers $P(\NN)\setminus \FF$.

Since $\FF$ is strong $P$, there is a sequence $k_0<k_1<k_2<\ldots$ of naturals such that
\begin{equation}\label{ast}
\text{if $X_n\in \CC_n$ for $n\in \NN$, then } \bigcup_n \left(X_n\cap[k_n,k_{n+1})\right)\in \FF  \tag{*}
\end{equation}

For $n\in\NN$ and $s\subseteq[k_n,k_{n+1})$, let 
\begin{equation}\label{astast}
U_s=\{A\in P(\NN):A\cap[k_n,k_{n+1})=s\}\tag{**}
\end{equation}

be the basic open set in $P(\NN)$ given by $s$.
For each $n\in \NN$ we set
$$W_n=\bigcup\{U_s:s\subseteq [k_n,k_{n+1})\text{ and } U_s\cap\CC_n\neq\emptyset\}.$$
The set $W_n$ is an open (in $P(\NN)$) neighborhood of $\CC_n$ and hence $P(\NN)\setminus W_n$ is a compact subset of $\bigcup \mathcal{U}_n$. It follows that there is a finite $\mathcal{V}_n\subseteq \mathcal{U}_n$ with $\bigcup\mathcal{V}_n\supseteq P(\NN)\setminus W_n$.

We claim that the family $\bigcup_{n\in \NN}\mathcal{V}_n$ covers $P(\NN)\setminus\FF$. Indeed, otherwise there is
$X\in P(\NN)\setminus\FF$ such that $X\notin \bigcup \mathcal{V}_n$, for every $n\in \NN$. This means that
$X\in W_n$, for every $n\in\NN$, so for some $s(n)\subseteq [k_n,k_{n+1})$ we have
$X\in U_{s(n)}$ and $U_{s(n)}\cap\CC_n\neq\emptyset$.
This implies that for every $n\in \NN$, there exists $X_n\in \CC_n$ such that $X_n\cap[k_n,k_{n+1})=s(n)=X\cap[k_n,k_{n+1})$.
We get $X=\bigcup\left(X_n\cap[k_n,k_{n+1})\right)\notin \FF$, contradicting \eqref{ast}.
This finishes the proof of $(1)\Rightarrow (2)$.

Suppose now that $K(\FF)$ is hereditarily Baire and fix a sequence $(\CC_n)_{n\in\NN}$ of compact subsets of $\FF$.
Without loss of generality we may assume that this sequence of compacta is increasing.
We shall define a strategy $\sigma$ for player I in the game $kP(P(\NN),\FF)$ together with a sequence $k^\sigma_0<k^\sigma_1<\ldots$ of natural numbers,
depending on the play according to $\sigma$,
in the following way:

We set $\sigma(\emptyset)=(\CC_0,P(\NN))$ and $k^\sigma_0=0$. Player II  replies and plays with a set $V_0$ open in $P(\NN)$ such that $\CC_0\subseteq V_0$.
Since $\CC_0$ is compact there are:
\begin{itemize}
    \item A natural number $k_1^\sigma>k_0^\sigma$,
    \item a finite set $F_0$ and
    \item sets $s^0_i\subseteq[k^\sigma_0,k^\sigma_1)$, for every $i\in F_0$
\end{itemize}
such that
$$W_0=\bigcup_{i\in F_0}U_{s^0_i}\subseteq V_0,\quad \CC_0\subseteq \bigcup_{i\in F_0} U_{s^0_i}\quad\text{and}\quad U_{s^0_i}\cap \CC_0\neq\emptyset \quad\text{for every } i\in F_0,$$
where $U_{s^0_i}$ is a basic open set in $P(\NN)$ given by $s^0_i$, cf. \eqref{astast}.
Let $$\KK_1=\{X\in \FF: X\in W_0\; \text{and}\; (\exists Y\in \CC_1)\; Y\cap[k_1^\sigma,\infty)=X\cap[k^\sigma_1,\infty)\}.$$
It is straightforward to check that $\KK_1$ is a nonempty (because $\CC_0\subseteq \CC_1$) compact subset of $\FF$.
We define $\sigma(V_0)=(\KK_1,W_0).$
Now, player II picks an open set $V_1$ such that $K_1\subseteq V_1\subseteq W_0$. As before,
there is a natural number $k^\sigma_2>k^\sigma_1$, a finite set $F_1$ and sets $s^1_i\subseteq [k^\sigma_1,k^\sigma_2)$, for $i\in F_1$, such that
$$W_1=W_0\cap \bigcup_{i\in F_1}U_{s^1_i}\subseteq V_1,\quad
\KK_1\subseteq W_1\quad\text{and}\quad U_{s^1_i}\cap \CC_1\neq\emptyset \quad\text{for every } i\in F_1,$$
where $U_{s^1_i}$ is a basic open set in $P(\NN)$ given by $s^1_i$. We set
$$\KK_2=\{X\in\FF: X\in W_1 \text{ and }(\exists Y\in \CC_2)\; Y\cap[k_2^\sigma,\infty)=X\cap[k_2,\infty)\}.$$
Again, this is a nonempty compact subset of $\FF$ so we can put
$\sigma(V_0,V_1)=(\KK_2,W_1)$. The game follows this pattern. In that way, we define a strategy $\sigma$ for player I.
By our assumption $K(\FF)$ is hereditarily Baire and hence $\sigma$ is not a winning strategy (cf. Theorem \ref{main}), i.e. there is a play 
$$\sigma(\emptyset),\;V_0,\; \sigma(V_0),\;V_1,\;\sigma(V_0,V_1),\;V_2,\;\ldots$$
in which player I applies her strategy and fails, so $\bigcap W_n\subseteq \FF$. In addition, this play generates a sequence $0=k_0<k_1<k_2<\ldots$
of natural numbers.
Let us show that this sequence must satisfy condition \eqref{ast}. For $n\in \NN$, take $X_n\in \CC_n$. It follows from definition of sets
$W_n$ that
$$\bigcup_n \left(X_n\cap[k_n,k_{n+1})\right)\in \bigcap_n W_n\subseteq \FF $$ and thus we have \eqref{ast}.
\end{proof}

\section{Spaces of measures}

For a topological space $X$, by $P(X)$ we denote the space of all Borel probability measures on $X$ endowed with the weak topology, i.e.
basic neighborhoods of $\mu\in P(X)$ are of the form
$$V_\mu(f_1,\ldots,f_n;\varepsilon)=\{\nu\in P(X):\left| \int f_i d\mu-\int f_i d\nu \right|<\varepsilon,\;i=1,\ldots,n\},$$
where $\varepsilon>0$ and $f_1,\ldots,f_n$ are real-valued bounded continuous functions on $X$.

A measure $\mu\in P(X)$ is called \textit{Radon} if for each Borel set $B\subseteq X$ and every $\varepsilon>0$ there is a compact set $K\subseteq B$
satisfying $\mu(B\setminus K)<\varepsilon$. The subspace of $P(X)$ consisting of all Radon measures will be denoted by $P_r(X)$.

It is known that $X$ is separable metrizable if and only if $P_r(X)$ is separable metrizable if and only if $P(X)$ is separable metrizable. Similarly,
$X$ is Polish if and only if $P_r(X)$ is Polish if and only if $P(X)$ is Polish. It is not clear however when $P_r(X)$ or $P(X)$ is (hereditarily) Baire.
Hereditary Baireness and other completeness type properties of spaces of measures on separable metrizable spaces were investigated, e.g. in Brown \cite{B} and
Brown \& Cox \cite{BC}, where several interesting examples are given -- some of them under additional set-theoretic assumptions.
We contribute to this line of research by adding the hyperspace to the picture
and proving the following:

\begin{thrm}\label{measures}
Let $X$ be a separable metrizable space. If the hyperspace $K(X)$ is hereditarily Baire, then
the space $P_r(X)$ is hereditarily Baire too, and thus $P(X)$ is Baire.
\end{thrm}

\begin{proof}
As usual, $2^{<\omega}$ is the set of all finite $0-1$ sequences. For $s\in 2^{<\omega}$, by $|s|$ we denote the length of $s$. If $k\in \NN$ then
$2^{\leq k}=\{s\in 2^{<\omega}:|s|\leq k\}$ and $2^k=\{s\in 2^{<\omega}:|s|=k\}$.

Let $Z$ be a metric compactification of $X$. The space $P(Z)=P_r(Z)$ of probability measures on $Z$ is metrizable and hence there is a
metric $d(\cdot,\cdot)$ on $P(Z)$ that generates the topology.
The space $P_r(X)$ is the subspace of $P(Z)$ consisting of all measures $\mu\in P(Z)$ for which $Z\setminus X$ has $\mu$-measure zero,
i.e. $\mu(B)=0$, for some Borel set $B\supseteq Z\setminus X$.

Striving for a contradiction, suppose that $Q=\{ \mu_1,\mu_2,\ldots \}$ is a closed copy of the rationals in $P_r(X)$, cf. Theorem \ref{thrm_Hurewicz}.
Recursively, for each $i\geq 0$, we construct:
\begin{itemize}
    \item a strategy $\sigma_i$ for player I in the $k$-Porada game $kP(Z,X)$,
    \item a set of measures $M_i\subseteq Q$
    \item a set $W_i$ open in $P(Z)$,
    \item a compact set $C_i\subseteq X$
    \end{itemize}
in such a way that, for every $i\geq 1$ the following conditions hold:
\begin{enumerate}
\item $M_i$ is dense-in-itself
\item $M_{i-1}\supseteq M_{i}$
\item $M_i\subseteq W_i$
\item $\{\mu_1,\ldots , \mu_i\}\cap \overline{W}_i=\emptyset$, where the closure is taken in the space $P(Z)$.
\item $W_{i-1}\supseteq \overline{W}_{i}$, where the closure is taken in the space $P(Z)$.
\item $(\forall \mu \in M_i)\;\;\mu(C_i)\geq 1-\frac{1}{2^{i}}$,
\end{enumerate}
Let $M_0=Q$, $W_0=P(Z)$ and $C_0=\emptyset$. Let $\sigma_0$ be arbitrary.
Fix $n\geq 1$ and suppose that $M_{n-1}$ and $W_{n-1}$ are already defined, and conditions (1)--(6) hold for all $i\leq n-1$.
We shall construct a strategy $\sigma_n$ along with the sets $M_n, W_n$ and $C_n$ in such a way that conditions (1)-(6) shall remain true for $i=n$.
To this end, pick two distinct measures $\nu_0,\nu_1\in M_{n-1}\setminus \{\mu_1,\ldots ,\mu_n\}$
(this is possible since by (1) the set $M_{n-1}$ has no isolated points).
Let $W_n$ be an open set in $P(Z)$ containing $\{\nu_0,\nu_1\}$ and satisfying
$$\overline{W}_n\subseteq W_{n-1} \quad\text{and}\quad \overline{W}_n\cap\{\mu_1,\ldots ,\mu_n\}=\emptyset.$$
(such set exists because $\nu_0,\nu_1\in M_{n-1}\setminus \{\mu_1,\ldots ,\mu_n\}\subseteq W_{n-1}$, by (3)).

Now, we shall inductively construct a strategy $\sigma_n$ along with a set of measures $$M_n=\{\nu_s\in M_{n-1}:s\in 2^{<\omega}\}.$$
Fix $k\geq 1$ and
suppose that the strategy $\sigma_n$ is defined up to the $(k-1)$-st move, i.e. we have a sequence
$$V_0=Z,
\;(K_1,U_1),\; V_1,\; (K_2,U_2),\;V_2, \ldots ,\;(K_{k-1},U_{k-1}),\;V_{k-1}$$
and a set of measures
$$M_n^k=\{\nu_s\in M_{n-1}:s\in 2^{<\omega}, 1\leq |s|\leq k\}$$
where
\begin{enumerate}[(i)]
    \item $K_j\subseteq V_j\subseteq U_j$, for every $1\leq j\leq k-1$;
    \item $K_{j+1}\subseteq U_j\subseteq \overline{U}_j\subseteq V_j$, for every $1\leq j\leq k-1$;
    \item $K_j\subseteq K_{j+1}$, for every $1\leq j\leq k-2$;
    \item For every $t\in 2^{k-1}$ we have $\nu_{t^\frown 0}=\nu_t$ and $\nu_{t^\frown 1}\neq\nu_t$ with $d(\nu_{t^\frown 1},\nu_t)<\frac{1}{2^{k-1}}$;
    \item $\nu_s(K_j)>1-2^{-(n+1)}-\ldots - 2^{-(n+j)}$ for every $1\leq j\leq k-1$ and $s\in 2^{\leq j+1}$.
\end{enumerate}
Let $$\widetilde{V}_{k-1}=\{\mu\in P(Z):\mu(V_{k-1})>1-2^{-(n+1)}-\ldots - 2^{-(n+k-1)}\}\cap W_n.$$
Note that if $s\in 2^{\leq k}$, then 
$\nu_s\in \widetilde{V}_{k-1}$
(this follows from (v) if $k>1$; for $k=1$ this is clear because $V_0=Z$).

For each $s\in 2^{k}$ we define
\begin{align*}
&\nu_{s^\frown 0}=\nu_s\quad\text{and}\\
&\nu_{s^\frown 1}\in M_{n-1}\cap \widetilde{V}_{k-1},\;\nu_{s^\frown 1}\neq\nu_s \text{ is an arbitrary
measure with } d(\nu_{s^\frown 1},\nu_s)<\tfrac{1}{2^{k}}
\end{align*}
(we can pick such $\nu_{s^\frown 1}$ since $\widetilde{V}_{k-1}$ is an open neighborhood of $\nu_s\in M_{n-1}$ and
$M_{n-1}$ has no isolated points by (1)).

Now, let $K_{k}\subseteq X\cap V_{k-1}$ be a compact set such that
$$K_{k-1}\subseteq K_k\quad\text{and}\quad \nu_s(K_k)>1-2^{-{n+1}}-\ldots-2^{-(n+k-1)}-2^{-(n+k)},$$
for every $s\in 2^{[k+1]}$.

Let $U_k$ be an open set in $Z$ satisfying 
  $$K_k\subseteq U_k\subseteq \overline{U}_k\subseteq V_{k-1}$$

We set
$$\sigma_n(V_0,\ldots, V_{k-1})=(K_k,U_k)\quad\text{and}\quad M_{n}^{k+1}=\{\nu_s:s\in 2^{<\omega}, 1\leq |s|\leq k+1\}$$
It is straightforward to check that conditions (i)-(v) hold. This ends the inductive construction of the strategy $\sigma_n$ and the set $$M_n=
\bigcup_{k=1}^\infty M^k_n.$$
By our assumption the space $K(X)$ is hereditarily Baire and hence player I cannot win if playing according to the strategy $\sigma_n$.
Since $V_k\subseteq U_k\subseteq \overline{U}_k\subseteq V_{k-1}$, by (ii), we must have
$$\emptyset\neq\bigcap V_k=\bigcap\overline{V}_k\subseteq X.$$
We define
$$C_n=\bigcap V_k.$$
It is clear that conditions (2)--(5) remain true for $i=n$ and condition (1) follows from (iv). Let us show (6).

We have $M_n=\{\nu_s:s\in 2^{<\omega}\}$. Fix $s\in 2^{<\omega}$ and let $m=|s|$.
By (v), $$\nu_s(K_m)>1-2^{-(n+1)}-\ldots - 2^{-(n+m)}>1-\frac{1}{2^n}.$$
Since $K_1\subseteq K_2\subseteq\ldots$ (by (iii)) we get $K_m\subseteq \bigcup_{k=1}^\infty K_k\subseteq \bigcap V_k=C_n$, so $\nu_s(C_n)>1-\frac{1}{2^n}$, as promised.
This finishes the construction of sets $M_n$, $W_n$ and $C_n$.

Let $\lambda_1,\lambda_2,\ldots$ be a one-to-one sequence of measures such that $\lambda_n\in M_n$, for every $n\geq1$. Let $\lambda$ be a complete accumulation point of $\{\lambda_n:n\geq 1\}$ in $P(Z)$ (it exists by compactness of $P(Z)$). Since $\lambda_n\in M_n$, it follows from (2) and (6) that $$\lambda_n (C_i)\geq 1-\frac{1}{2^i},\quad \text{for every }n\geq i$$
and hence $\lambda(\bigcup C_i)=1$. But $\bigcup C_i\subseteq X$ thus $\lambda\in P_r(X)$.
On the other hand, $\lambda_n\in M_n\subseteq W_n$ and hence it follows from (4) and (5) that
$\lambda\notin Q$. Therefore $\lambda\in \overline{Q}\setminus Q\subseteq P(Z)\setminus P_r(X)$
(where the last inclusion follows form the assumption that $Q$ is closed in $P_r(X)$). This yields a contradiction.
\end{proof}

We do not know if analogous result is true for the property of Baire.

\begin{question}\label{q1}
Let $X$ be metrizable and suppose that $K(X)$ is Baire. Is it true that $P_r(X)$ is Baire?
\end{question}

Since $P_r(X)$ is a dense subspace of $P(X)$, Baireness of $P_r(X)$ implies Baireness of $P(X)$. The property of Baire in hyperspaces $K(X)$ of all compact and
$CL(X)$ of all closed subsets of a topological space $X$,
has been a subject of intense study, see
e.g. McCoy \cite{MC}, Cao \& Tomita \cite{CT}, Cao et al. \cite{CGG}.

One can ask a more general question.
\begin{question}
Let $X$ be metrizable and suppose that $CL(X)$ is Baire. Is it true that $P(X)$ is Baire?
\end{question}
The affirmative answer would be a generalization of the main result of \cite{CGG} for metrizable spaces. Indeed, W\'ojcicka proved in \cite{W} that
$X^n$ is Baire for
every $n$, provided $X$ is metrizable and $P(X)$ is Baire.

The class of spaces with hereditarily Baire spaces of measures, and hence by Theorem \ref{measures}, also the class of spaces with hereditarily Baire
hyperspaces is somewhat related with the important class of Prohorov spaces.
Recall that a topological space $X$ is \textit{Prohorov} if every compact set in $P_r(X)$ is uniformly tight, i.e. for every compact $C\subseteq P_r(X)$ and every
$\varepsilon>0$, there is a compact set $K\subseteq X$ such that $\mu(K)>1-\varepsilon$, for every $\mu\in C$.
The following observation is known.
\begin{prop}
If a metrizable space $X$ is Prohorov, then $P_r(X)$ is hereditarily Baire.
\end{prop}
\begin{proof}
Since $X$ is metrizable the space $P_r(X)$ is metrizable too.
By \cite[Theorem 1]{W1} we infer that $P_r(X)$ is Prohorov. The Prohorov property is inherited by closed subspaces and since $\mathbb{Q}$ is not Prohorov \cite{P}
the result follows from Theorem
\ref{thrm_Hurewicz}.
\end{proof}

Every completely metrizable space is Prohorov (Prohorov's theorem)
and one of the intriguing open questions is whether there is a ZFC example of a universally measurable $X\subseteq [0,1]$
which is Prohorov but is not completely metrizable (see \cite[page 225]{Bo}.
In light of Corollary \ref{maincor2} it would be natural to check whether the complement of a Menger non-$\sigma$-compact subset of $[0,1]$ is Prohorov.
It is worth mentioning that Tsaban and Zdomskyy \cite[Section 2.5]{TZ} identified the whole class of Menger non-$\sigma$-compact subsets of the Cantor set
$2^\mathbb{N}$.

No topological description of Prohorov spaces is known. However, since the Prohorov property is invariant under images and preimages by perfect maps, see \cite{Top}
one can try to find a description in terms of the remainder of a compactification of $X$.
Let us take a closer look at this approach.

For a topological space $X$, let us denote by $\mathcal{O}$ (resp. $\mathcal{O}_k$, $\mathcal{O}^\ast$) the family of all open covers
(respectively, all open $k$-covers\footnote{A family $\mathcal{U}$ of subsets of $X$ is called a \textit{$k$-cover} of $X$ provided for each compact $K\subseteq X$,
there is an element $U\in \mathcal{U}$ with $K\subseteq U$.}, all open covers closed under finite unions).
Let $X$ be a metric space and let $\mathcal{C}$ be a collection of open covers of $X$. Let us denote by $G_1^X(\mathcal{C},\mathcal{O})$ the following game played
by two players: Player one picks open covers $\mathcal{U}_0, \mathcal{U}_1,\ldots$ from $\mathcal{C}$ and player II chooses sets $U_0,U_1,\ldots$ where
$U_n\in \mathcal{U}_n$, for every $n$. We declare that player II wins the run of the game if the family $\{U_n:n=0,1,\ldots\}$ belongs to $\mathcal{O}$, i.e.
covers $X$. Otherwise, player I wins.

It is not difficult to see that the Menger game $M(X)$ described in paragraph 1.2 is equivalent to the game $G_1^X(\mathcal{O}^\ast,\mathcal{O})$. Thus, if
$X\subseteq Z$, where $Z$ is a metric space,
player I has no winning strategy in $G_1^{Z\setminus X}(\mathcal{O}^\ast,\mathcal{O})$ if and only if $K(X)$ is hereditarily Baire. Since
$\mathcal{O}^\ast\subseteq \mathcal{O}_k$, if player I has no winning strategy in the game $G_1^Y(\mathcal{O}_k,\mathcal{O})$, she has no winning strategy in
the game $G_1^Y(\mathcal{O}^\ast,\mathcal{O})$, for any topological space $Y$.

Let $X$ be a separable metric space and let $bX$ be a metric compactification of $X$.
It was shown recently by F.\ Jordan in \cite{J} that player I has no winning strategy in the game $G_1^{bX\setminus X}(\mathcal{O}_k,\mathcal{O})$
if and only if $X$
is consonant (see \cite[Section 3]{HP} for more information on consonant spaces). Therefore, it follows from \cite[Proposition 4]{Bou}, that if player I has no
winning strategy in $G_1^{bX\setminus X}(\mathcal{O}_k,\mathcal{O})$, then the hyperspace $K(X)$ is Prohorov (in particular $X$ is Prohorov). As
we noted above, if player I has no winning strategy in $G_1^{bX\setminus X}(\mathcal{O}_k,\mathcal{O})$, then player I has no winning strategy in
$G_1^{bX\setminus X}(\mathcal{O}^\ast,\mathcal{O})$ which means that $K(X)$ is hereditarily Baire. It is natural to ask the following:
\begin{question}
Suppose that a metric space $X$ is Prohorov. Is it true that $K(X)$ is hereditarily Baire? 
\end{question}

\bibliography{games1a}

\begin{thebibliography}{10}

\bibitem{AB}
{\sc L.~F. Aurichi and A.~Bella}, {\em When is a space {M}enger at infinity?},
  Appl. Gen. Topol., 16 (2015), pp.~75--80.

\bibitem{AD}
{\sc L.~F. Aurichi and R.~R. Dias}, {\em A minicourse on topological games},
  Topology Appl., 258 (2019), pp.~305--335.

\bibitem{BT}
{\sc T.~Bartoszy\'{n}ski and B.~Tsaban}, {\em Hereditary topological
  diagonalizations and the {M}enger-{H}urewicz conjectures}, Proc. Amer. Math.
  Soc., 134 (2006), pp.~605--615.

\bibitem{BH-G}
{\sc A.~Bella and R.~Hern\'{a}ndez-Guti\'{e}rrez}, {\em A non-discrete space
  {$X$} with {$C_p(X)$} {M}enger at infinity}, Appl. Gen. Topol., 20 (2019),
  pp.~223--230.

\bibitem{BZT}
{\sc A.~Bella, S.~Tokg\"{o}z, and L.~Zdomskyy}, {\em Menger remainders of
  topological groups}, Arch. Math. Logic, 55 (2016), pp.~767--784.

\bibitem{BHV}
{\sc A.~Blass, M.~Hru\v{s}\'{a}k, and J.~Verner}, {\em On strong {$P$}-points},
  Proc. Amer. Math. Soc., 141 (2013), pp.~2875--2883.

\bibitem{Bo}
{\sc V.~I. Bogachev}, {\em Measure theory. {V}ol. {II}}, Springer-Verlag,
  Berlin, 2007.

\bibitem{Bou}
{\sc A.~Bouziad}, {\em A note on consonance of {$G_\delta$} subsets}, Topology
  Appl., 87 (1998), pp.~53--61.

\bibitem{B}
{\sc J.~B. Brown}, {\em Baire category in spaces of probability measures},
  Fund. Math., 96 (1977), pp.~189--193.

\bibitem{BC}
{\sc J.~B. Brown and G.~V. Cox}, {\em Baire category in spaces of probability
  measures. {II}}, Fund. Math., 121 (1984), pp.~143--148.

\bibitem{CGG}
{\sc J.~Cao, S.~Garc\'{\i}a-Ferreira, and V.~Gutev}, {\em Baire spaces and
  {V}ietoris hyperspaces}, Proc. Amer. Math. Soc., 135 (2007), pp.~299--303.

\bibitem{CT}
{\sc J.~Cao and A.~H. Tomita}, {\em Baire spaces, {T}ychonoff powers and the
  {V}ietoris topology}, Proc. Amer. Math. Soc., 135 (2007), pp.~1565--1573.

\bibitem{CP}
{\sc J.~Chaber and R.~Pol}, {\em On hereditarily {B}aire spaces,
  {$\sigma$}-fragmentability of mappings and {N}amioka property}, Topology
  Appl., 151 (2005), pp.~132--143.

\bibitem{Debs}
{\sc G.~Debs}, {\em Espaces h\'{e}r\'{e}ditairement de {B}aire}, Fund. Math.,
  129 (1988), pp.~199--206.

\bibitem{GMZ}
{\sc P.~Gartside, A.~Medini, and L.~Zdomskyy}, {\em The {V}ietoris hyperspace
  $\mathcal{K}({X})$ is hereditarily {B}aire if and only if ${X}$ is
  co-{M}enger}, preprint,  (2018).

\bibitem{GHM}
{\sc O.~Guzm\'{a}n, M.~Hru\v{s}\'{a}k, and A.~Mart\'{\i}nez-Celis}, {\em Canjar
  filters}, Notre Dame J. Form. Log., 58 (2017), pp.~79--95.

\bibitem{HI}
{\sc M.~Henriksen and J.~R. Isbell}, {\em Some properties of
  compactifications}, Duke Math. J., 25 (1958), pp.~83--105.

\bibitem{HP}
{\sc v.~Hol\'{a} and J.~Pelant}, {\em Recent progress in hyperspace
  topologies}, in Recent progress in general topology, {II}, North-Holland,
  Amsterdam, 2002, pp.~253--285.

\bibitem{Hu2}
{\sc W.~{Hurewicz}}, {\em {\"Uber eine Verallgemeinerung des Borelschen
  Theorems.}}, {Math. Z.}, 24 (1926), pp.~401--421.

\bibitem{Hu1}
\leavevmode\vrule height 2pt depth -1.6pt width 23pt, {\em {Relativ perfekte
  Teile von Punktmengen und Mengen \((A)\)}}, {Fundam. Math.}, 12 (1928),
  pp.~78--109.

\bibitem{J}
{\sc F.~Jordan}, {\em Consonant spaces and topological games}, Topology Appl.,
  274 (2020), p.~107121.

\bibitem{Kechris}
{\sc A.~S. Kechris}, {\em Classical descriptive set theory}, vol.~156 of
  Graduate Texts in Mathematics, Springer-Verlag, New York, 1995.

\bibitem{La}
{\sc C.~Laflamme}, {\em Forcing with filters and complete combinatorics}, Ann.
  Pure Appl. Logic, 42 (1989), pp.~125--163.

\bibitem{Marciszewski}
{\sc W.~Marciszewski}, {\em {$P$}-filters and hereditary {B}aire function
  spaces}, Topology Appl., 89 (1998), pp.~241--247.

\bibitem{MC}
{\sc R.~A. McCoy}, {\em Baire spaces and hyperspaces}, Pacific J. Math., 58
  (1975), pp.~133--142.

\bibitem{FM}
{\sc A.~W. Miller and D.~H. Fremlin}, {\em On some properties of {H}urewicz,
  {M}enger, and {R}othberger}, Fund. Math., 129 (1988), pp.~17--33.

\bibitem{Po}
{\sc E.~Porada}, {\em Jeu de {C}hoquet}, Colloq. Math., 42 (1979),
  pp.~345--353.

\bibitem{P}
{\sc D.~Preiss}, {\em Metric spaces in which {P}rohorov's theorem is not
  valid}, Z. Wahrscheinlichkeitstheorie und Verw. Gebiete, 27 (1973),
  pp.~109--116.

\bibitem{Scheepers}
{\sc M.~Scheepers}, {\em Combinatorics of open covers. {I}. {R}amsey theory},
  Topology Appl., 69 (1996), pp.~31--62.

\bibitem{T}
{\sc R.~Telg\'{a}rsky}, {\em On games of {T}ops\o e}, Math. Scand., 54 (1984),
  pp.~170--176.

\bibitem{Telgarsky}
\leavevmode\vrule height 2pt depth -1.6pt width 23pt, {\em Remarks on a game of
  {C}hoquet}, Colloq. Math., 51 (1987), pp.~365--372.

\bibitem{Top}
{\sc F.~Tops{\o}e}, {\em Compactness and tightness in a space of measures with
  the topology of weak convergence}, Math. Scand., 34 (1974), pp.~187--210.

\bibitem{TZ}
{\sc B.~Tsaban and L.~Zdomskyy}, {\em Scales, fields, and a problem of
  {H}urewicz}, J. Eur. Math. Soc. (JEMS), 10 (2008), pp.~837--866.

\bibitem{vanDouwen}
{\sc E.~K. van Douwen}, {\em Closed copies of the rationals}, Comment. Math.
  Univ. Carolin., 28 (1987), pp.~137--139.

\bibitem{W}
{\sc M.~W\'{o}jcicka}, {\em Note on the {B}aire category in spaces of
  probability measures on nonseparable metrizable spaces}, Bull. Polish Acad.
  Sci. Math., 33 (1985), pp.~305--311.

\bibitem{W1}
\leavevmode\vrule height 2pt depth -1.6pt width 23pt, {\em The space of
  probability measures on a {P}rohorov space is {P}rohorov}, Bull. Polish Acad.
  Sci. Math., 35 (1987), pp.~809--811.

\end{thebibliography}
\bibliographystyle{siam}

\end{document}